\documentclass{birkjour}

\usepackage{url}
\newtheorem{thm}{Theorem}[section]

\newtheorem{lem}[thm]{Lemma}

\newtheorem{conj}[thm]{Conjecture}
\theoremstyle{definition}

\theoremstyle{remark}

\numberwithin{equation}{section}

\makeatletter
\long\def\bj@first#1#2\bj@nil{#1}
\long\def\bj@second#1#2#3\bj@nil{#2}
\def\bj@ref#1#2{%
  \expandafter\ifx\csname r@#2\endcsname\relax
    \protect\G@refundefinedtrue
    \nfss@text{\reset@font\bfseries ??}%
    \@latex@warning{Reference `#2' on page \thepage\space undefined}%
  \else
    \expandafter\expandafter\expandafter#1\csname r@#2\endcsname\bj@nil
  \fi}
\renewcommand{\ref}[1]{{\rm\bj@ref\bj@first{#1}}}
\renewcommand{\pageref}[1]{{\upshape\bj@ref\bj@second{#1}}}
\makeatother

\newcommand{\Z}{\mathbb{Z}}
\newcommand{\N}{\mathbb{N}}
\newcommand{\one}{\mathbf{1}}
\newcommand{\chips}[1]{\lvert #1\rvert}

\allowdisplaybreaks

\begin{document}

%

\title[Chip-firing games with period two]
 {All parallel chip-firing games with\\
  $2|E|-|V|<\chips{\sigma}<2|E|$ have period $2$}

\author[D. Wang]{Daniel Wang}
\address{Yale University\br
New Haven, CT 06520\br
USA}
\email{daniel.wang.dzw22@yale.edu}

\author[N. Lannan]{Nathan Lannan}
\email{nathan@nathanlannan.com}




\begin{abstract}
In 2010, Levine found that the activity functions of parallel chip-firing games
on the complete graph $K_n$ converge to a devil's staircase pattern: adding chips
causes activity to progress through open intervals in which it is locally
constant. In 2022, Bu, Choi, and Xu improved on an earlier bound by Kominers and
Kominers to show that there exists a strict lower bound below which all games
have activity $0$, and a strict upper bound above which all games have activity
$1$. They thereby generalized the bottom and topmost rungs of the devil's
staircase to all graphs. In 2024, Ji, Li, and Wang conjectured that a similarly
general bound exists for games with activity $\tfrac12$. We use GPT-5.6-Sol to
prove this conjecture, unifying existing results for trees, cycles, complete
graphs, and complete bipartite graphs. This generalizes the middle rung of the
devil's staircase.
\end{abstract}

\maketitle

\section{Introduction}

The \emph{parallel chip-firing game} is an automaton on a finite connected simple
graph $G=(V,E)$. At the beginning of the game, each vertex $v\in V$ holds some
nonnegative number of chips. In each round, all vertices holding at least as many
chips as they have neighbors simultaneously \emph{fire} one chip to each of their
neighbors. All other vertices \emph{wait}. A \emph{chip configuration} is a map
$\sigma:V\to\N_0$ recording the number of chips on each vertex, and we write
$\chips{\sigma}=\sum_{v\in V}\sigma(v)$ for the total number of chips in the
game. Since chips are never created or destroyed, $\chips{\sigma}$ is invariant
under the dynamics, and the system's evolution is completely determined by $G$
and $\sigma$.

Bitar and Goles~\cite{BitarGoles1992} showed that every parallel chip-firing game
is eventually periodic, with some minimal period length $T$. Games with $T=1$ are
said to \emph{stabilize}. Kominers and Kominers~\cite{KominersKominers2010}
proved that every game with $\chips{\sigma}\ge 4|E|-|V|$ stabilizes. Bu, Choi,
and Xu~\cite{BuChoiXu2022} made this bound exact, showing that games with
$\chips{\sigma}<|E|$ or $\chips{\sigma}>3|E|-|V|$ stabilize, while for every
$\chips{\sigma}$ with $|E|\le\chips{\sigma}\le 3|E|-|V|$ there exist
non-stabilizing games.

Jiang~\cite{Jiang2010} showed that all vertices fire an equal number of times
within a period. The \emph{activity} of a parallel chip-firing game is then the
fraction of rounds in which a vertex fires over a complete period.
Levine~\cite{Levine2011} observed a \emph{devil's staircase} pattern in the
activity of parallel chip-firing games on complete graphs $K_n$---in the limit of
large $n$, activity is locally constant over open intervals of
$\chips{\sigma}/n$. He described the wide stairs as \emph{mode-locking}, since
adding even a relatively large number of chips does not change the activity in
these regimes. Later, Kiss, Levine, and
T\'othm\'er\'esz~\cite{KissLevineTothmeresz2025} developed parallel chip-firing on graphons and proved that, under mild assumptions, the devil's staircase applies to sequences of Erdős–Rényi random graphs. We can view the work of Bu, Choi, and Xu~\cite{BuChoiXu2022}
as generalizing the bottom and top stairs, on which the activity is respectively
$0$ and $1$, to all graphs. 

Ji, Li, and Wang~\cite{JiLiWang2024} recently initiated the study of the middle
stair, on which the activity is $\tfrac12$; equivalently, the period is $2$. They
proved that no parallel chip-firing game with
$2|E|-|V|<\chips{\sigma}<2|E|$ has $T=3$ or $T=4$, and conjectured that the
middle stair occupies exactly that band of chip counts.

\begin{conj}[{\cite[Conjecture~2]{JiLiWang2024}}]\label{conj:jlw}
Any parallel chip-firing game on $G=(V,E)$ with $2|E|-|V|<\chips{\sigma}<2|E|$
has $T=2$.
\end{conj}

Conjecture~\ref{conj:jlw} has separately been established for trees by Bitar and
Goles~\cite{BitarGoles1992}, for cycles by Dall'Asta~\cite{DallAsta2006}, for
complete graphs by Levine~\cite{Levine2011}, for balanced complete bipartite graphs
$K_{a,a}$ by Ji, Li, and Wang~\cite{JiLiWang2024}, and for all complete bipartite graphs $K_{a,b}$ by Jung~\cite{Jung2026}. The
chip band appearing in it is exactly the band in which all games are
non-stabilizing.

In this paper, we present a recent proof of Conjecture~\ref{conj:jlw} by
GPT-5.6-Sol. We thereby generalize the middle stair of Levine's devil's
staircase~\cite{Levine2011} to all graphs.

\begin{thm}\label{thm:main}
Let $G=(V,E)$ be a finite connected simple graph and let $\sigma:V\to\N_0$ be the
initial chip configuration of a parallel chip-firing game on $G$. If
$2|E|-|V|<\chips{\sigma}<2|E|$, then the game has $T=2$ and activity $\tfrac12$.
\end{thm}

\section{The setting}\label{sec:setting}

Throughout, $G=(V,E)$ is a finite connected simple graph with $n=|V|$, $m=|E|$,
and $d_v=\deg(v)$, and $\sigma:V\to\N_0$ is a chip configuration.

We define the \emph{step operator} $U$, where $U\sigma$ is the chip configuration
obtained after one round of firing on $\sigma$. We let $U^0\sigma=\sigma$ and
$U^t\sigma=U(U^{t-1}\sigma)$, and we let $\sigma_t=U^t\sigma$ be the chip
configuration after round $t$. We define $F_t(v)$ to be $1$ if $v$ fires on round
$t$ and $0$ if $v$ waits, so that
\begin{equation}\label{eq:indicator}
   F_t(v)=
   \begin{cases}
     1,&\sigma_t(v)\ge d_v,\\
     0,&\sigma_t(v)<d_v.
   \end{cases}
\end{equation}

We consider $F_t\in\{0,1\}^V$ and $\sigma_t\in\N_0^V$ as vectors. Let $L=D-A$ be
the Laplacian matrix of $G$, where $D$ is the diagonal degree matrix and $A$ the
adjacency matrix. Since each vertex gains chips equal to the number of neighbors
that fire, and loses chips equal to its degree, we can write
\begin{equation}\label{eq:update}
   \sigma_{t+1}=\sigma_t-DF_t+AF_t=\sigma_t-LF_t.
\end{equation}

Since $F_t$ can be decomposed into the standard basis $\{e_i\}$ and $Le_i$ has
elements summing to $0$ for each $e_i$, this equation shows that
$\chips{\sigma_{t+1}}=\chips{\sigma_t}$.

We define the \emph{period} of a parallel chip-firing game to be the smallest $T$
such that there exists a round $t$ with $U^t\sigma=U^{t+T}\sigma$; such a $T$
exists because the configurations lie in a finite set~\cite{BitarGoles1992}. We
denote the smallest such $t$ as $t_0$. For simplicity, we replace $\sigma$ by
$U^{t_0}\sigma$; equivalently, we assume $t_0=0$. In other words, we consider
only the periodic portion of a parallel chip-firing game.

For each $0\le t\le T$, we let
\begin{equation}\label{eq:firings}
   u_v(t)=\sum_{s=0}^{t-1}F_s(v)
\end{equation}
denote the number of rounds before $t$ on which $v$ fires. Letting $u(t)$ be a
vector, we telescope equation~\eqref{eq:update} to obtain
\begin{equation}\label{eq:telescoping}
   \sigma_t=\sigma-Lu(t).
\end{equation}
Jiang~\cite{Jiang2010} showed that all vertices fire the same number of times in
a period.

\begin{lem}[{\cite[Proposition~2.5]{Jiang2010}}]\label{lem:common-count}
Let $\sigma$ be a chip configuration and $T\ge1$. Then $U^T\sigma=\sigma$ if and
only if $u_u(T)=u_v(T)$ for all $u,v\in V$.
\end{lem}

We call the value $q=u_v(T)$ the \emph{common firing count} of the configuration
$\sigma$; note that $0\le q\le T$. Every vertex therefore fires in the same
fraction $q/T$ of the rounds of a cycle, and this common proportion is the
\emph{activity} of the configuration $a(\sigma)$.

The \emph{firing sequence} of $v$ is the binary word
$F_0(v)F_1(v)\cdots F_{T-1}(v)$, with indices taken modulo $T$. Jiang, Scully,
and Zhang~\cite{JiangScullyZhang2015} defined a firing sequence to be
\emph{clumpy} if it contains both $00$ and $11$ as cyclic substrings, and showed
that clumpy firing sequences do not occur.

\begin{lem}[{\cite[Theorem~6.2]{JiangScullyZhang2015}}]\label{thm:nonclumpy}
No vertex in a parallel chip-firing game has a firing sequence containing both
$00$ and $11$ as cyclic substrings.
\end{lem}

A vertex $v$ is \emph{abundant} on round $t$ if $\sigma_t(v)\ge 2d_v$. We define
the \emph{complement} of a game on $G$ with chip configuration $\sigma$ to be the
game on $G$ with chip configuration
\begin{equation}\label{eq:complement}
   \sigma^c(v)=2d_v-1-\sigma(v).
\end{equation}
This is a chip configuration precisely when no vertex is abundant. Kominers and
Kominers~\cite{KominersKominers2010} established that abundant vertices do not
exist in non-stabilizing games, so the complement game exists and
$\chips{\sigma^c}+\chips{\sigma}=\sum_{v\in V}(2d_v-1)=4|E|-|V|$. Let $F^c_t(v)$
be $1$ if $v$ fires on round $t$ in the complement game and $0$ if $v$ waits.
Jiang~\cite{Jiang2010} showed that complement games stay complements in later
rounds, and that a vertex fires in the complement exactly when it waits in the
original.

\begin{lem}[{\cite[Lemma~2.3]{Jiang2010}}]\label{lem:comp}
For any configuration $\sigma$ with a complement configuration $\sigma^c$, it
holds that $U\sigma^c=(U\sigma)^c$. Furthermore, $F^c_t(v)=1-F_t(v)$.
\end{lem}

It follows that the period of a game equals the period of its complement, and
that $a(\sigma)=1-a(\sigma^c)$.

\section{Main result}\label{sec:main}

Let $\sigma$ be a chip configuration on $G$ with period $T$ and activity
$a(\sigma)=q/T<1$. We define the \emph{prefix height} of a vertex $v$ as
\begin{equation}\label{eq:prefix}
   h_v(t)=T\,u_v(t)-qt,\qquad 0\le t\le T.
\end{equation}

The height measures how far ahead of its average firing rate $q/T$ the vertex $v$
is at round $t$. Note that since $a(\sigma)=q/T$, it holds for every $v\in V$
that $h_v(T)=h_v(0)=0$, and furthermore that $h_v(t+T)=h_v(t)$ for all $t$. The
following lemma uses the prefix height of $v$ to find a round on which $v$ waits.

\begin{lem}\label{lem:selection}
Let $\sigma$ be a chip configuration with period $T$ and activity
$a(\sigma)=q/T<1$. Then every vertex $v$ admits a round
$\tau_v\in\{0,\ldots,T-1\}$ such that
\begin{equation}\label{eq:selected-waits}
   F_{\tau_v}(v)=0
\end{equation}
and, for all $0\le t\le T-1$, we have
\begin{equation}\label{eq:selected-min}
   h_v(\tau_v)-q\le h_v(t).
\end{equation}
\end{lem}

\begin{proof}
Choose a phase $t_v^*\in\{0,\ldots,T-1\}$ at which $h_v$ attains its minimum over
$\Z/T\Z$, and let $\tau_v\equiv t_v^*-1 \pmod T$. By equation~\eqref{eq:prefix},
\[
   0\ge h_v(\tau_v+1)-h_v(\tau_v)=T\,F_{\tau_v}(v)-q.
\]
Since $q<T$, we must have $F_{\tau_v}(v)<1$; hence $F_{\tau_v}(v)=0$. Then
$u_v(\tau_v+1)=u_v(\tau_v)$, so equation~\eqref{eq:prefix} gives
$h_v(\tau_v+1)=h_v(\tau_v)-q$. Combining with $h_v(\tau_v+1)\le h_v(t)$ for all
$t$, we get equation~\eqref{eq:selected-min}.
\end{proof}

We fix such a round $\tau_v$ for each vertex $v$; we call this round the
\emph{waiting round} of $v$. For an edge $e=\{v,w\}\in E$, we define the
\emph{edge charge} of $e$ as
\begin{equation}\label{eq:edge-charge}
   K_{vw}=\bigl(u_v(\tau_v)-u_w(\tau_v)\bigr)+\bigl(u_w(\tau_w)-u_v(\tau_w)\bigr).
\end{equation}

We show that for $a(\sigma)<\tfrac12$, the edge charge for every edge is
nonpositive.

\begin{lem}\label{lem:edge-charge}
If $2q<T$, then $K_{vw}\le 0$ for every edge $\{v,w\}\in E$.
\end{lem}

\begin{proof}
The terms $qt$ cancel in the difference
$h_v(t)-h_w(t)=T\bigl(u_v(t)-u_w(t)\bigr)$, so rearranging
equation~\eqref{eq:edge-charge} gives
\begin{equation}\label{eq:K-as-h}
   T\,K_{vw}
   =\bigl(h_v(\tau_v)-h_v(\tau_w)\bigr)+\bigl(h_w(\tau_w)-h_w(\tau_v)\bigr).
\end{equation}
Applying Lemma~\ref{lem:selection} for $v$ gives
$h_v(\tau_v)-h_v(\tau_w)\le q$, and similarly we have
$h_w(\tau_w)-h_w(\tau_v)\le q$. Substituting both into equation~\eqref{eq:K-as-h} yields
$T\,K_{vw}\le 2q$. Since $K_{vw}$ is an integer and $2q<T$, we must have
$K_{vw}\le 0$.
\end{proof}

We can bound the chip count of a configuration using the sum of edge charges.

\begin{lem}\label{lem:summation}
For any chip configuration $\sigma$ with $a(\sigma)=q/T<1$, it holds that
\[
   \chips{\sigma}\le 2|E|-|V|+\sum_{\{v,w\}\in E}K_{vw}.
\]
\end{lem}

\begin{proof}
Since $v$ waits on round $\tau_v$, we have $\sigma_{\tau_v}(v)\le d_v-1$.
Evaluating equation~\eqref{eq:telescoping} at the vertex-dependent round $\tau_v$ and
reading off the $v$-th coordinate,
\begin{align*}
   \sigma(v)&=\sigma_{\tau_v}(v)+\bigl(Lu(\tau_v)\bigr)_v\\
   &\le d_v-1+\bigl(Du(\tau_v)\bigr)_v-\bigl(Au(\tau_v)\bigr)_v.
\end{align*}
Note that $\bigl(Du(\tau_v)\bigr)_v=d_vu_v(\tau_v)$, and
$\bigl(Au(\tau_v)\bigr)_v=\sum_{\{v,w\}\in E}u_w(\tau_v)$. Thus we can write
\[
   \sigma(v)\le d_v-1+\sum_{\{v,w\}\in E}\bigl(u_v(\tau_v)-u_w(\tau_v)\bigr).
\]
Summing over $v\in V$ and grouping the terms corresponding to like edges, each
edge $\{v,w\}$ contributes
$\bigl(u_v(\tau_v)-u_w(\tau_v)\bigr)+\bigl(u_w(\tau_w)-u_v(\tau_w)\bigr)=K_{vw}$,
so that
\begin{align*}
   \chips{\sigma}
   &\le\sum_{v\in V}(d_v-1)+\sum_{\{v,w\}\in E}K_{vw}\\
   &=2|E|-|V|+\sum_{\{v,w\}\in E}K_{vw},
\end{align*}
as claimed.
\end{proof}

Combining Lemmas~\ref{lem:edge-charge} and~\ref{lem:summation} bounds the chip
count of any configuration with activity less than $\tfrac12$.

\begin{lem}\label{lem:low-density}
For any chip configuration $\sigma$ with $a(\sigma)=q/T<\tfrac12$, it holds that
\[
   \chips{\sigma}\le 2|E|-|V|.
\]
\end{lem}

\begin{proof}
By Lemma~\ref{lem:summation},
$\chips{\sigma}\le 2|E|-|V|+\sum_{\{v,w\}\in E}K_{vw}$. Every $K_{vw}$ is
nonpositive by Lemma~\ref{lem:edge-charge}.
\end{proof}

We can use complement configurations to obtain the corresponding bound for
high-activity games.

\begin{lem}\label{lem:high-density}
For any chip configuration $\sigma$ with $\tfrac12<a(\sigma)=q/T$, it holds that
\[
   \chips{\sigma}\ge 2|E|.
\]
\end{lem}

\begin{proof}
If $a(\sigma)=1$, then $\sigma(v)\ge d_v$ for all $v$ and
$\chips{\sigma}\ge 2|E|$. Otherwise the game is nonstabilizing, so the complement
configuration $\sigma^c$ exists and has activity $a(\sigma^c)=1-a(\sigma)<\tfrac12$
by Lemma~\ref{lem:comp}. Lemma~\ref{lem:low-density} therefore gives
$\chips{\sigma^c}\le 2|E|-|V|$. Since
$\chips{\sigma^c}=4|E|-|V|-\chips{\sigma}$, rearranging yields
$\chips{\sigma}\ge 2|E|$.
\end{proof}

The bounds of Lemmas~\ref{lem:low-density} and~\ref{lem:high-density} are tight.
On every connected graph with at least two vertices, the stable configuration
$\sigma(v)=d_v-1$ has $\chips{\sigma}=2|E|-|V|$ and $a(\sigma)=0$. Similarly, the
stable configuration $\sigma(v)=d_v$ yields $\chips{\sigma}=2|E|$ and
$a(\sigma)=1$.

We are now ready to prove Theorem~\ref{thm:main}. We continue to assume that
$\sigma$ is a chip configuration on $G$ with period $T$ and activity
$a(\sigma)=q/T$. Let $\sigma$ satisfy $2|E|-|V|<\chips{\sigma}<2|E|$.

\begin{proof}[Proof of Theorem~\ref{thm:main}]
We first show that $a(\sigma)=\tfrac12$ for all chip configurations within this
bound. If $2q<T$, Lemma~\ref{lem:low-density} gives
$\chips{\sigma}\le 2|E|-|V|$, contradicting the lower bound. If $2q>T$, then
Lemma~\ref{lem:high-density} gives $\chips{\sigma}\ge 2|E|$, contradicting the
upper bound. Hence $2q=T$. In particular $T$ is even, and the activity of the
game is $a(\sigma)=q/T=\tfrac12$.

We now show that $T=2$. Fix $v\in V$. Its firing sequence
$F_0(v)\cdots F_{T-1}(v)$ is a cyclic binary word of length $T=2q$ containing
exactly $q$ ones and $q$ zeros. By Lemma~\ref{thm:nonclumpy}, it is not clumpy,
so it must have no $00$ or no $11$. If it has no $00$, then the $q$ zeros divide
the $q$ ones into individual nonempty blocks; if it has no $11$ we can make the
same argument in reverse. Thus the firing sequence must alternate, so
$F_{t+1}(v)=1-F_t(v)$ for all $t$. As this holds for all $v$, we must have
$F_{t+1}=\one-F_t$ for all $t$. Applying equation~\eqref{eq:update} twice,
\[
   \sigma_{t+2}=\sigma_t-L(F_t+F_{t+1})=\sigma_t-L\one=\sigma_t
\]
for every $t\ge0$, where $L\one=\mathbf{0}$ is the zero vector because every row
of $L$ sums to $0$. Hence $U^{2}\sigma=\sigma$, so $2$ is a return length of
$\sigma$ and the period $T$ divides $2$. Since $T$ is even, we must have $T=2$.
\end{proof}


\section*{Acknowledgments}

The authors thank David Ji and Michael Li for assisting with reviewing the proof. We also thank Lionel Levine for helpful correspondence.


\section*{Use of large language models}

The proof presented in Section~\ref{sec:main} was found by the large language
model GPT-5.6-Sol. The authors verified the resulting argument and take full
responsibility for the correctness of the mathematics and for the exposition.



\begin{thebibliography}{9}

\bibitem{BitarGoles1992} Bitar, J., Goles, E.: Parallel chip firing games on
graphs. Theor. Comput. Sci. \textbf{92}, 291--300 (1992).
\url{https://doi.org/10.1016/0304-3975(92)90316-8}

\bibitem{BuChoiXu2022} Bu, A., Choi, Y., Xu, M.: An exact bound on the number of
chips of parallel chip-firing games that stabilize. Arch. Math. \textbf{119},
471--478 (2022). \url{https://doi.org/10.1007/s00013-022-01777-3}

\bibitem{DallAsta2006} Dall'Asta, L.: Exact solution of the one-dimensional
deterministic fixed-energy sandpile. Phys. Rev. Lett. \textbf{96}, 058003
(2006). \url{https://doi.org/10.1103/PhysRevLett.96.058003}

\bibitem{JiLiWang2024} Ji, D., Li, M., Wang, D.: Non-stabilizing parallel
chip-firing games. Preprint, arXiv:2408.10508 (2024).

\bibitem{Jiang2010} Jiang, T.-Y.: On the period lengths of the parallel
chip-firing game. Preprint, arXiv:1003.0943 (2010).

\bibitem{JiangScullyZhang2015} Jiang, T.-Y., Scully, Z., Zhang, Y.X.: Motors and
impossible firing patterns in the parallel chip-firing game. SIAM J. Discret.
Math. \textbf{29}, 615--630 (2015). \url{https://doi.org/10.1137/130933770}

\bibitem{Jung2026} Jung, M.: The middle stair for complete bipartite parallel
chip-firing. Preprint, arXiv:2608.01350 (2026).

\bibitem{KissLevineTothmeresz2025} Kiss, V., Levine, L., T\'othm\'er\'esz, L.:
The devil's staircase for chip-firing on random graphs and on graphons. Random
Struct. Algorithms \textbf{66}, e21255 (2025).
\url{https://doi.org/10.1002/rsa.21255}

\bibitem{KominersKominers2010} Kominers, P.M., Kominers, S.D.: A constant bound
for the periods of parallel chip-firing games with many chips. Arch. Math.
\textbf{95}, 9--13 (2010). \url{https://doi.org/10.1007/s00013-010-0129-x}

\bibitem{Levine2011} Levine, L.: Parallel chip-firing on the complete graph:
Devil's staircase and Poincar\'e rotation number. Ergod. Theory Dyn. Syst.
\textbf{31}, 891--910 (2011). \url{https://doi.org/10.1017/S0143385710000088}

\end{thebibliography}
\end{document}